\documentclass[12pt]{article}
\usepackage{amsmath,amsthm,amscd,amssymb,amscd}
\DeclareMathOperator{\Grass}{Grass}
\DeclareMathOperator{\Fil}{Fil}

\DeclareMathOperator{\Sym}{Sym}

\DeclareMathOperator{\Mat}{Mat}
\DeclareMathOperator{\Spec}{Spec}
\DeclareMathOperator{\Ker}{Ker}

\DeclareMathOperator{\codim}{codim}
\DeclareMathOperator{\rank}{rank}
\DeclareMathOperator{\Sing}{Sing}

\newtheorem{theorem}{Theorem}[section]

\newtheorem{lemma}[theorem]{Lemma}
\newtheorem{proposition}[theorem]{Proposition}
\theoremstyle{definition}

\theoremstyle{remark}

\numberwithin{equation}{section}
\title{The varieties of tangent lines to hypersurfaces\\
 in projective spaces}
\author{Atsushi Ikeda}
\date{}
\begin{document}
\maketitle
\begin{abstract}
 For a hypersurface in a projective space, we consider the set of pairs
 of a point and a line in the projective space such that the line
 intersects the hypersurface at the point with a fixed multiplicity.
 We prove that this set of pairs forms a smooth variety for a general
 hypersurface.
\end{abstract}
\section{Introduction}
Let $\mathbf{P}^{n}$ be the projective space of dimension $n$ over a
field $K$.
We denote by $X_{F}$ the hypersurface in $\mathbf{P}^{n}$ defined by a
homogeneous polynomial $F\in{K[x_{0},\dots,x_{n}]}$ of degree
$d$.
Let $\mathbf{G}$ be the Grassmannian variety of all lines in
$\mathbf{P}^{n}$.
Then the set
$$
Z_{F}=\{L\in{\mathbf{G}}\mid
\text{$L\subset{X_{F}}$}\}
$$
forms a closed subscheme of $\mathbf{G}$, and it is called
{\itshape Fano scheme} of lines in $X_{F}$.
The Fano schemes for cubic threefolds were first studied by Fano, and
they were used by Tjurin \cite{t} and Clemens-Griffiths \cite{cg} in the
proof of the Torelli theorem and the irrationality for cubic threefolds
over the complex numbers.
Then the foundations of the Fano schemes of cubic hypersurfaces for any
characteristic were given by Altman-Kleiman \cite{ak}, and the results
on the smoothness and connectedness of $Z_{F}$ for any degree $d$ were
proved by Barth-Van~de~Ven \cite{bv} and bettered in the book
\cite[Chapter~V.~4]{k} by Koll\'{a}r.
In this paper, we introduce the following scheme $Y_{F,m}$ as an
analogy of the Fano scheme $Z_{F}$.
For $1\leq{m}\leq{\infty}$, we set
$$
Y_{F,m}=\{(p,L)\in\mathbf{P}^{n}\times{\mathbf{G}}\mid
\text{$L$ intersects $X_{F}$ at $p$ with the multiplicity
$\geq{m}$}\},
$$
which forms a closed subscheme of $\mathbf{P}^{n}\times\mathbf{G}$.
Since $Y_{F,1}$ is a $\mathbf{P}^{n-1}$-bundle over $X_{F}$ by the first
projection and $Y_{F,\infty}$ is a $\mathbf{P}^{1}$-bundle over $Z_{F}$
by the second projection, the scheme $Y_{F,m}$ is considered to be an
intermediate object between $X_{F}$ and $Z_{F}$.
We expect to characterize some geometric properties of $X_{F}$ by using
the Hodge structure of $Y_{F,m}$.
A computation for the Hodge structure of $Y_{F,m}$ is announced in the
summary \cite{i}.\par
In Section~$\ref{vi}$, following the formulation for the Fano schemes
in \cite{ak}, we define the scheme $Y_{F,m}$ as the zeros of a
section of a vector bundle on a flag variety.
It enables us to compute the Chern numbers of $Y_{F,m}$ by Schubert
calculus.
In Section~$\ref{sc}$, we investigate the smoothness and connectedness
of $Y_{F,m}$ for $m\leq{d}$.
If $m\leq2n-1$ and $m$ is prime to the characteristic of $K$, then
$Y_{F,m}$ is smooth of dimension $2n-m-1$ for a general hypersurface
$X_{F}$ (Theorem~$\ref{gensm}$. $(\mathrm{iii})$).
If $m\leq2n-2$, then $Y_{F,m}$ is connected for any hypersurface $X_{F}$
(Theorem~$\ref{gensm}$. $(\mathrm{iv})$).
Particularly for a cubic hypersurface $X_{F}$, the variety $Y_{F,3}$ is
smooth of dimension $2n-4$ if and only if $X_{F}$ is a smooth
hypersurface (Theorem~\ref{cub}).
These results for $Y_{F,m}$ proved in Section~$\ref{sc}$ corresponds to
the results for the Fano scheme $Z_{F}$ proved in \cite{bv} and
\cite[Chapter~V.~4]{k}.
\section{Varieties of pairs of a point and a line}\label{vi}
Let $\mathbf{P}^{n}=\mathbf{P}^{n}_{K}$ be the projective space of
dimension $n$ over a field $K$,
and let $V$ be the $K$-vector space
$H^{0}(\mathbf{P}^{n},\mathcal{O}_{\mathbf{P}^{n}}(1))$.
We denote by $\mathbf{P}=\Grass{(n,V)}$ the Grassmannian variety of
all $n$-dimensional subspaces in $V$, and denote by
$\mathcal{Q}_{\mathbf{P}}$ the universal quotient bundle on
$\mathbf{P}$.
Then $\mathbf{P}$ is naturally identified with $\mathbf{P}^{n}$, and
$\mathcal{Q}_{\mathbf{P}}$ is identified with the tautological line
bundle $\mathcal{O}_{\mathbf{P}^{n}}(1)$.
We denote by $\mathbf{G}=\Grass{(n-1,V)}$ the Grassmannian variety of
all $(n-1)$-dimensional subspaces in $V$, and denote by
$\mathcal{Q}_{\mathbf{G}}$ the universal quotient bundle on
$\mathbf{G}$.
We remark that a point of $\mathbf{G}$ corresponds to a line in
$\mathbf{P}^{n}$.
Let $\Gamma\subset\mathbf{P}\times\mathbf{G}$ be
the flag variety of all pairs $(p,L)$ of a point $p\in\mathbf{P}^{n}$
and a line $L\subset\mathbf{P}^{n}$ containing the point $p$.
The variety $\Gamma$ is the $\mathbf{P}^{n-1}$-bundle over $\mathbf{P}$
by the first projection $\phi:\Gamma\rightarrow\mathbf{P}$, and
$\Gamma$ is the $\mathbf{P}^{1}$-bundle over $\mathbf{G}$
by the second projection $\pi:\Gamma\rightarrow\mathbf{G}$.
We define the line bundle $\mathcal{Q}_{\phi}$ on $\Gamma$ as
the kernel of the natural surjective homomorphism
$\pi^{*}\mathcal{Q}_{\mathbf{G}}
\rightarrow
\phi^{*}\mathcal{Q}_{\mathbf{P}}$,
and define a decreasing filtration
$$
\Sym^{d}\pi^{*}\mathcal{Q}_{\mathbf{G}}
=\Fil^{0}\Sym^{d}\pi^{*}\mathcal{Q}_{\mathbf{G}}
\supset\dots\supset
\Fil^{d}\Sym^{d}\pi^{*}\mathcal{Q}_{\mathbf{G}}
\supset
\Fil^{\infty}\Sym^{d}\pi^{*}\mathcal{Q}_{\mathbf{G}}=0
$$
on the $d$-th symmetric product of $\pi^{*}\mathcal{Q}_{\mathbf{G}}$,
as $\Fil^{m}\Sym^{d}\pi^{*}\mathcal{Q}_{\mathbf{G}}$
being the image of the natural homomorphism
$$
\Sym^{m}\mathcal{Q}_{\phi}
\otimes\Sym^{d-m}\pi^{*}\mathcal{Q}_{\mathbf{G}}
\longrightarrow
\Sym^{d}\pi^{*}\mathcal{Q}_{\mathbf{G}}
$$
for $0\leq{m}\leq{d}$, and
$\Fil^{\infty}\Sym^{d}\pi^{*}\mathcal{Q}_{\mathbf{G}}=0$.
Let $F\in\Sym^{d}V$.
We denote by $X_{F}$ the hypersurface in
$\mathbf{P}$ defined as the zeros of the section
$[F]_{\mathbf{P}}\in
{H^{0}(\mathbf{P},\Sym^{d}\mathcal{Q}_{\mathbf{P}})}$
which is the image of $F$ by the natural isomorphism
$$
\Sym^{d}V\simeq
H^{0}(\mathbf{P},\Sym^{d}\mathcal{Q}_{\mathbf{P}}).
$$
We denote by $Z_{F}$ the subscheme in $\mathbf{G}$ defined as the
zeros of the section
$[F]_{\mathbf{G}}\in
H^{0}(\mathbf{G},\Sym^{d}\mathcal{Q}_{\mathbf{G}})$
which is the image of $F$ by the natural isomorphism
$$
\Sym^{d}V\simeq
H^{0}(\mathbf{G},\Sym^{d}\mathcal{Q}_{\mathbf{G}}).
$$
Then a point in $Z_{F}$ corresponds to a line contained in
$X_{F}$, and $Z_{F}$ is called the {\itshape Fano scheme} of lines in
$X_{F}$.
We denote by $Y_{F,m}$ the subscheme in $\Gamma$ defined as the
zeros of the section
$[F]_{\Gamma,m}\in
H^{0}(\Gamma,\Sym^{d}\pi^{*}\mathcal{Q}_{\mathbf{G}}/
\Fil^{m}\Sym^{d}\pi^{*}\mathcal{Q}_{\mathbf{G}})$
which is the image of $F$ by the natural homomorphism
$$
\Sym^{d}V\simeq
H^{0}(\Gamma,\Sym^{d}\pi^{*}\mathcal{Q}_{\mathbf{G}})
\longrightarrow
H^{0}(\Gamma,\Sym^{d}\pi^{*}\mathcal{Q}_{\mathbf{G}}/
\Fil^{m}\Sym^{d}\pi^{*}\mathcal{Q}_{\mathbf{G}}).
$$
Let $L$ be a line in $\mathbf{P}^{n}$, and let $p$ be a point on $L$.
The fiber of the line bundle $\mathcal{Q}_{\phi}$ at
$(p,L)\in\Gamma$ is identified with the kernel of the
restriction
$$
H^{0}(L,\mathcal{O}_{\mathbf{P}^{n}}(1)\vert_{L})
\longrightarrow
H^{0}(p,\mathcal{O}_{\mathbf{P}^{n}}(1)\vert_{p}).
$$
Hence, $L$ intersects $X_{F}$ at $p$ with the multiplicity $\geq{m}$
if and only if the pair $(p,L)$ represents a point in $Y_{F,m}$.
We have a diagram
$$
\begin{CD}
 p@.\quad\in\quad@.\mathbf{P}@.\supset@.{X_{F}}\\
 @AAA@.@AA{\phi}A@.@AA{\phi\vert_{Y_{F,1}}}A\\
 (p,L)@.\quad\in\quad@.
 \Gamma@.
 \quad\supset\quad@.{Y_{F,1}}@.\quad\supset\quad@.{Y_{F,2}}
 @.\quad\supset\quad@.\cdots@.\quad\supset\quad@.{Y_{F,d}}@.
 \quad\supset\quad@.{Y_{F,\infty}}\\
 @VVV@.@VV{\pi}V@.@.@.@.@.@.@.@.@.@VV{\pi\vert_{Y_{F,\infty}}}V\\
 L@.\quad\in\quad@.\mathbf{G}@.@.@.@.@.
 \supset@.@.@.@.@.{Z_{F}}.
\end{CD}
$$
The morphism 
$
\phi\vert_{Y_{F,1}}:Y_{F,1}\rightarrow{X_{F}}
$
is the $\mathbf{P}^{n-1}$-bundle, whose fiber at $p\in{X_{F}}$
is the set of all lines through the point $p$.
If $X_{F}$ is a smooth hypersurface, then
$
\phi\vert_{Y_{F,2}}:Y_{F,2}\rightarrow{X_{F}}
$
is the $\mathbf{P}^{n-2}$-bundle, whose fiber at $p\in{X_{F}}$
is the set of all lines through the point $p$ and contained in the
projective tangent space of $X_{F}$ at $p$.
The morphism
$
\pi\vert_{Y_{F,\infty}}:Y_{F,\infty}\rightarrow{Z_{F}}
$
is the $\mathbf{P}^{1}$-bundle, whose fiber at $L\in{Z_{F}}$
is the set of all points on the line $L$.\par
For $(p,L)\in\Gamma$, there is a basis $(x_{0},\dots,x_{n})$ of $V$
such that the point $p$ is defined by $x_{1}=\cdots=x_{n}=0$ and the
line $L$ is defined by $x_{2}=\cdots=x_{n}=0$ in $\mathbf{P}^{n}$.
Then the map
$$
\begin{matrix}
 \mathbf{A}^{2n-1}
 =&\Spec{K[\xi_{1},\dots,\xi_{n},\zeta_{2},\dots,\zeta_{n}]}
 &\overset{\sim}{\longrightarrow}&{U}\subset\Gamma;\\
 &(\xi_{1},\dots,\xi_{n},\zeta_{2},\dots,\zeta_{n})
 &\longmapsto&(p_{\xi},L_{(\xi,\zeta)})
\end{matrix}
$$
gives a local coordinate of $\Gamma$ at $(p,L)$, where $p_{\xi}$
denotes the point defined by
$$
x_{1}-\xi_{1}x_{0}=\cdots=x_{n}-\xi_{n}x_{0}=0
$$
and $L_{(\xi,\zeta)}$ denotes the line defined by
$$
(x_{2}-\xi_{2}x_{0})-\zeta_{2}(x_{1}-\xi_{1}x_{0})=
\cdots=(x_{n}-\xi_{n}x_{0})-\zeta_{n}(x_{1}-\xi_{1}x_{0})=0.
$$
On this local coordinate $U$,
$([x_{0}]_{U},[x_{1}]_{U})$ is a
local basis of $\pi^{*}\mathcal{Q}_{\mathbf{G}}$, and
$[x_{1}-\xi_{1}x_{0}]_{U}$
is a local basis of $\mathcal{Q}_{\phi}$, where $[A]_{U}$
denotes the image of $A\in{V}$ by the restriction
$V\rightarrow{H^{0}(U,\pi^{*}\mathcal{Q}_{\mathbf{G}})}$.
Note that $([x_{0}]_{U},[x_{1}-\xi_{1}x_{0}]_{U})$ is another local
basis of $\pi^{*}\mathcal{Q}_{\mathbf{G}}$.
We define the polynomial
$f_{k}(\xi,\zeta)=
f_{k}(\xi_{1},\dots,\xi_{n},\zeta_{2},\dots,\zeta_{n})$
by
$$
[F]_{U}=\sum_{k=0}^{d}f_{k}(\xi,\zeta)
[x_{1}-\xi_{1}x_{0}]_{U}^{k}[x_{0}]_{U}^{d-k}
\in{H^{0}(U,\Sym^{d}\pi^{*}\mathcal{Q}_{\mathbf{G}})}.
$$
Then $[F]_{U}$ is contained in
${H^{0}(U,\Fil^{m}\Sym^{d}\pi^{*}\mathcal{Q}_{\mathbf{G}})}$
if and only if
$$
f_{0}(\xi,\zeta)=\cdots=f_{m-1}(\xi,\zeta)=0,
$$
and we have
$$
Y_{F,m}\cap{U}\simeq
\Spec{K[\xi_{1},\dots,\xi_{n},\zeta_{2},\dots,\zeta_{n}]\big/
\bigl(f_{0}(\xi,\zeta),\dots,f_{m-1}(\xi,\zeta)\bigr)}.
$$
When we consider $F\in\Sym^{d}V$ as the homogeneous polynomial
$F(x_{0},\dots,x_{n})\in{K[x_{0},\dots,x_{n}]}$
of degree $d$, we have
\begin{multline*}
 F(x_{0},x_{1},\zeta_{2}(x_{1}-\xi_{1}x_{0})+\xi_{2}x_{0},\dots,
 \zeta_{n}(x_{1}-\xi_{1}x_{0})+\xi_{n}x_{0})\\
 =\sum_{k=0}^{d}f_{k}(\xi,\zeta)
 (x_{1}-\xi_{1}x_{0})^{k}x_{0}^{d-k},
\end{multline*}
hence the local equation of $X_{F}\cap{L_{(\xi,\zeta)}}$ in
$L_{(\xi,\zeta)}$ is
\begin{equation}\label{le}
 F(1,t+\xi_{1},\zeta_{2}t+\xi_{2},\dots,\zeta_{n}t+\xi_{n})
  =\sum_{k=0}^{d}f_{k}(\xi,\zeta)t^{k},
\end{equation}
where $t=\frac{x_{1}}{x_{0}}-\xi_{1}$ is a local parameter of the line
$L_{(\xi,\zeta)}$ at the point $p_{\xi}$.
\section{Smoothness and connectedness}\label{sc}
Since $Y_{F,\infty}$ is a $\mathbf{P}^{1}$-bundle over $Z_{F}$,
the following theorem is directly induced from the results in
\cite[Theorem~8]{bv} and \cite[Chapter~V.~Theorem~4.3]{k}.
\begin{theorem}
 Assume $d\geq1$.
 \begin{enumerate}
  \item If $d\geq2n-2$, then $Y_{F,\infty}$ is empty for general
	$F\in\Sym^{d}V$.
  \item If $d\leq2n-3$, then $Y_{F,\infty}$ is non-empty for any
	$F\in\Sym^{d}V$.
  \item If $d\leq2n-3$, then $Y_{F,\infty}$ is smooth of dimension $2n-d-2$
	for general $F\in\Sym^{d}V$.
  \item If $d\leq{2n-4}$ and $(d,n)\neq(2,3)$, then $Y_{F,\infty}$ is
	connected for any $F\in\Sym^{d}V$.
 \end{enumerate}
\end{theorem}
In this section, we prove the following theorem;
\begin{theorem}\label{gensm}
 Assume $1\leq{m}\leq{d}$.
 \begin{enumerate}
  \item If $m\geq2n$, then $Y_{F,m}$ is empty for general
	$F\in\Sym^{d}V$.
  \item If $m\leq2n-1$, then $Y_{F,m}$ is non-empty for any
	$F\in\Sym^{d}V$.
  \item If $m\leq2n-1$ and $m$ is prime to the
	characteristic of $K$,
	then $Y_{F,m}$ is smooth of dimension $2n-m-1$ for general
	$F\in\Sym^{d}V$.
  \item If $m\leq{2n-2}$, then $Y_{F,m}$ is connected for any
	$F\in\Sym^{d}V$.
 \end{enumerate}
\end{theorem}
We denote by $M_{d}=\Grass{(1,\Sym^{d}V)}$ the space of hypersurfaces
of degree $d$ in $\mathbf{P}^{n}$.
We set the vector bundle $\mathcal{E}_{d,m}$ on $\Gamma$ by
$$
\mathcal{E}_{d,m}
=\Ker{(\mathcal{O}_{\Gamma}\otimes\Sym^{d}V\longrightarrow
\Sym^{d}\pi^{*}\mathcal{Q}_{\mathbf{G}}/
\Fil^{m}\Sym^{d}\pi^{*}\mathcal{Q}_{\mathbf{G}})},
$$
and we consider the projective space bundle
$Y_{d,m}=\Grass{(1,\mathcal{E}_{d,m})}
\rightarrow\Gamma$.
Then $Y_{d,m}$ is a smooth subvariety of codimension $m$ in
$M_{d}\times{\Gamma}=
\Grass{(1,\mathcal{O}_{\Gamma}\otimes\Sym^{d}V)}$,
and the fiber of the projection
$$
\psi_{d,m}:Y_{d,m}\longrightarrow{M_{d}};\
([F],p,L)\longmapsto[F]
$$
at $[F]\in{M_{d}}$ is equal to $Y_{F,m}$.
We denote by ${Y_{d,m}(p,L)}\subset{M_{d}}$ the fiber of the projective
space bundle $Y_{d,m}\rightarrow\Gamma$ at $(p,L)\in\Gamma$.
For $(p,L)\in\Gamma$, we fix a basis $(x_{0},\dots,x_{n})$ of $V$
such that the point $p$ is defined by $x_{1}=\cdots=x_{n}=0$ and the
line $L$ is defined by $x_{2}=\cdots=x_{n}=0$ in $\mathbf{P}^{n}$.
Then $Y_{d,m}(p,L)$ is the linear subspace
$$
Y_{d,m}(p,L)=\{[F]\in{M_{d}}\mid
a_{0}=\cdots=a_{m-1}=0\},
$$
where $a_{i}$ denotes the coefficient of the monomial
$x_{0}^{d-i}x_{1}^{i}$ in $F(x_{0},\dots,x_{n})$.
For $(p,L)\in{Y_{m,F}}$, we define the matrix $J_{m}([F],p,L)$ by
$$
J_{m}([F],p,L)=
\begin{pmatrix}
 \frac{\partial{f_{0}}}{\partial{\xi_{1}}}(0)&\cdots&
 \frac{\partial{f_{0}}}{\partial{\xi_{n}}}(0)&
 \frac{\partial{f_{0}}}{\partial{\zeta_{2}}}(0)&\cdots&
 \frac{\partial{f_{0}}}{\partial{\zeta_{n}}}(0)\\
 &\cdots&&&\cdots&\\
 \frac{\partial{f_{m-1}}}{\partial{\xi_{1}}}(0)&\cdots&
 \frac{\partial{f_{m-1}}}{\partial{\xi_{n}}}(0)&
 \frac{\partial{f_{m-1}}}{\partial{\zeta_{2}}}(0)&\cdots&
 \frac{\partial{f_{m-1}}}{\partial{\zeta_{n}}}(0)
\end{pmatrix},
$$
where $(\xi_{1},\dots,\xi_{n},\zeta_{2},\dots,\zeta_{n})$ is the local
coordinate of $\Gamma$ and $f_{0}(\xi,\zeta),\dots,f_{m-1}(\xi,\zeta)$
are the local equations of $Y_{F,m}$ in Section~$\ref{vi}$.
By the equation $(\ref{le})$, we have
$$
 J_{m}([F],p,L)=
 \begin{pmatrix}
  0&a_{0,2}&\cdots&a_{0,n}&0&\cdots&0\\
  &&&&a_{0,2}&\cdots&a_{0,n}\\
  \vdots&&\cdots&&&&\\
  &&&&&\cdots&\\
  0&a_{m-2,2}&\cdots&a_{m-2,n}&&&\\
  ma_{m}&a_{m-1,2}&\cdots&a_{m-1,n}&a_{m-2,2}&\cdots&a_{m-2,n}
 \end{pmatrix},
$$
where $a_{k,j}$ denotes the coefficient of the monomial
$x_{0}^{d-k-1}x_{1}^{k}x_{j}$ in $F(x_{0},\dots,x_{n})$.
We define the degeneracy locus $W_{d,m}$ in $Y_{d,m}$ by
\begin{align*}
 W_{d,m}&=\{([F],p,L)\in{Y_{d,m}}\mid\rank{J_{m}([F],p,L)}<{m}\}\\
 &=\{([F],p,L)\in{Y_{d,m}}\mid\rank{\mathrm{d}\psi_{d,m}([F],p,L)}
 <\dim{M_{d}}\},
\end{align*}
where $\mathrm{d}\psi_{d,m}$ denotes the homomorphism on tangent spaces
induced by $\psi_{d,m}:Y_{d,m}\rightarrow{M_{d}}$.
We remark that $W_{d,2}\subset{W_{d,m}}$ for $m\geq2$, and we set
$W_{d,m}^{0}=W_{d,m}\setminus{W_{d,2}}$.
\begin{proposition}\label{cod}
 Assume $1\leq{m}\leq{d}$.
 \begin{enumerate}
  \item $\codim_{Y_{d,1}}{W_{d,1}}=n$ for $d\geq1$.
  \item $\codim_{Y_{d,2}}{W_{d,2}}=n-1$ for $d\geq2$.
  \item If $m=2n-1$ is prime to the characteristic of $K$, then
	$\codim_{Y_{d,2n-1}}{W_{d,2n-1}}=1$.
  \item If $3\leq{m}\leq{2n-2}$ and $m$ is prime to the characteristic
	of $K$, then
	$$
	\begin{cases}
	 \codim_{Y_{d,m}}{W_{d,m}}=\min{\{n-1,2n-m\}},\\
	 \codim_{Y_{d,m}}{W_{d,m}^{0}}=2n-m.
	\end{cases}
	$$
  \item If $3\leq{m}\leq{2n-2}$ and $m$ is divisible by the
	characteristic of $K$, then
	$$
	\begin{cases}
	 \codim_{Y_{d,m}}{W_{d,m}}=\min{\{n-1,2n-m-1\}},\\
	 \codim_{Y_{d,m}}{W_{d,m}^{0}}=2n-m-1.
	\end{cases}
	$$
 \end{enumerate}
\end{proposition}
We denote by $\Mat{(l,r)}$ the $K$-vector space of $l\times{r}$
matrices.
We define a subscheme $\Delta(l,r)$ in
$\Grass{(1,\Mat{(l,r)})}$ by
$$
\Delta(l,r)=\{[B]\in\Grass{(1,\Mat{(l,r)})}\mid
\rank{\tilde{B}}<l\},
$$
where we set
$$
\tilde{B}=
\begin{pmatrix}
 b_{1,1}&\cdots&b_{1,r}&0&\cdots&0\\
 &&&b_{1,1}&\cdots&b_{1,r}\\
 &\cdots&&&&\\
 &&&&\cdots&\\
 b_{l-1,1}&\cdots&b_{l-1,r}&&&\\
 b_{l,1}&\cdots&b_{l,r}&b_{l-1,1}&\cdots&b_{l-1,r}
\end{pmatrix}
\in\Mat{(l,2r)}
$$
for a matrix
$$
B=
\begin{pmatrix}
 b_{1,1}&\cdots&b_{1,r}\\
 &\cdots&\\
 b_{l,1}&\cdots&b_{l,r}
\end{pmatrix}
\in\Mat{(l,r)}.
$$
We set an open subset $\Delta^{0}(l,r)$ of $\Delta(l,r)$ by
$$
\Delta^{0}(l,r)=
\{[B]\in\Delta(l,r)\mid
(b_{1,1},\dots,b_{1,r})\neq(0,\dots,0)
\}.
$$
\begin{lemma}\label{delm}
 For $3\leq{l}\leq{2r}$,
 $$
 \begin{cases}
  \codim_{\Grass{(1,\Mat{(l,r)})}}{\Delta(l,r)}
  =\min{\{r,2r-l+1\}},\\
  \codim_{\Grass{(1,\Mat{(l,r)})}}{\Delta^{0}(l,r)}
  =2r-l+1.
 \end{cases}
 $$
\end{lemma}
\begin{proof}
 For $2\leq{i}\leq{l}$, we set
 $$
 \Delta_{i}(l,r)=
 \Bigl\{[B]\in\Delta(l,r)\mid
 \rank{
 \begin{pmatrix}
  b_{1,1}&\cdots&b_{1,r}&0&\cdots&0\\
  &&&b_{1,1}&\cdots&b_{1,r}\\
  &\cdots&&&&\\
  &&&&\cdots&\\
  b_{i-1,1}&\cdots&b_{i-1,r}&&&\\
  b_{i,1}&\cdots&b_{i,r}&b_{i-1,1}&\cdots&b_{i-1,r}
 \end{pmatrix}}
 <i\Bigr\}.
 $$
 Then
 $$
 \Delta_{2}(l,r)\subset\dots\subset\Delta_{l}(l,r)=\Delta(l,r),
 $$
 hence we have $\Delta(l,r)=\Delta_{2}(l,r)\amalg\Delta^{0}(l,r)$ and
 $$
 \Delta^{0}(l,r)=
 \coprod_{i=3}^{l}\bigl(\Delta_{i}(l,r)\setminus\Delta_{i-1}(l,r)\bigr).
 $$
 Since
 $$
 \Delta_{2}(l,r)=\{[B]\in\Delta(l,r)\mid
 (b_{1,1},\dots,b_{1,r})=(0,\dots,0)
 \},
 $$
 we have $\codim_{\Grass{(1,\Mat{(l,r)})}}{\Delta_{2}(l,r)}=r$.
 For $3\leq{i}\leq{l}$, there is an open immersion
 $$
 \Delta_{i}(l,r)
 \setminus\Delta_{i-1}(l,r)
 {\longrightarrow}
 \Grass{(1,\Mat{(l-2,r)})}\times\mathbf{A}^{i-1};
 $$
 $$
 \left[
 \begin{matrix}
  b_{1,1}&\cdots&b_{1,r}\\
  &\cdots&\\
  b_{l,1}&\cdots&b_{l,r}
 \end{matrix}
 \right]
 \longmapsto
 \Bigl(
 \left[
 \begin{matrix}
  b_{1,1}&\cdots&b_{1,r}\\
  &\cdots&\\
  b_{i-2,1}&\cdots&b_{i-2,r}\\
  b_{i+1,1}&\cdots&b_{i+1,r}\\
  &\cdots&\\
  b_{l,1}&\cdots&b_{l,r}
 \end{matrix}
 \right],
 (\alpha_{1},\dots,\alpha_{i-1})\Bigr),
 $$
 where $(\alpha_{1},\dots,\alpha_{i-1})$ is determined by
 \begin{multline*}
  (b_{i,1},\dots,b_{i,r},b_{i-1,1},\dots,b_{i-1,r})\\
  =(\alpha_{1},\dots,\alpha_{i-1})
  \begin{pmatrix}
   b_{1,1}&\cdots&b_{1,r}&0&\cdots&0\\
   &&&b_{1,1}&\cdots&b_{1,r}\\
   &\cdots&&&&\\
   &&&&\cdots&\\
   b_{i-2,1}&\cdots&b_{i-2,r}&&&\\
   b_{i-1,1}&\cdots&b_{i-1,r}&b_{i-2,1}&\cdots&b_{i-2,r}
  \end{pmatrix}.
 \end{multline*}
 Hence we have
 $
 \codim_{\Grass{(1,\Mat{(l,r)})}}{\bigl(\Delta_{i}(l,r)
 \setminus\Delta_{i-1}(l,r)\bigr)}=2r-i+1.
 $
\end{proof}
\begin{proof}[Proof of Proposition~$\ref{cod}$]
 We set $W_{d,m}(p,L)={Y_{d,m}(p,L)}{\cap}W_{d,m}$
 and $W_{d,m}^{0}(p,L)={Y_{d,m}(p,L)}\cap{W_{d,m}^{0}}$, and we compute
 their codimension in $Y_{d,m}(p,L)$. 
 It is clear that
 $\codim_{Y_{d,1}(p,L)}{W_{d,1}(p,L)}=n$ and
 $\codim_{Y_{d,2}(p,L)}{W_{d,2}(p,L)}=n-1$.
 Since $W_{d,2n-1}(p,L)$ is defined by $\det{J_{2n-1}([F],p,L)}=0$ in
 $Y_{d,2n-1}(p,L)$, if $m=2n-1$ is prime to the characteristic of $K$,
 then we have $\codim_{Y_{d,2n-1}(p,L)}{W_{d,2n-1}(p,L)}=1$.
 We assume $3\leq{m}\leq{2n-2}$.
 We define the hyperplane $T_{d,m}(p,L)$ in $Y_{d,m}(p,L)$ by
 $$
 T_{d,m}(p,L)=\{[F]\in{Y_{d,m}(p,L)}\mid
 a_{m}=0\}.
 $$
 If $m$ is prime to the characteristic of $K$, then by Lemma~$\ref{delm}$,
 we have
 \begin{align*}
  &\codim_{Y_{d,m}(p,L)}{(W_{d,m}(p,L)\setminus{T_{d,m}(p,L)})}\\
  &=\codim_{\Grass{(1,\Mat{(m-1,n-1)})}}{\Delta(m-1,n-1)}
  =\min{\{n-1,2n-m\}},\\
  &\codim_{Y_{d,m}(p,L)}{(W_{d,m}(p,L)\cap{T_{d,m}(p,L)})}\\
  &=1+\codim_{{Y_{d,m}(p,L)}\cap{T_{d,m}(p,L)}}
  {(W_{d,m}(p,L)\cap{T_{d,m}(p,L)})}\\
  &=1+\codim_{\Grass{(1,\Mat{(m,n-1)})}}{\Delta(m,n-1)}
  =\min{\{n,2n-m\}}
 \end{align*}
 and
 \begin{align*}
  &\codim_{Y_{d,m}(p,L)}{(W_{d,m}^{0}(p,L)\setminus{T_{d,m}(p,L)})}\\
  &=\codim_{\Grass{(1,\Mat{(m-1,n-1)})}}{\Delta^{0}(m-1,n-1)}
  =2n-m,\\
  &\codim_{Y_{d,m}(p,L)}{(W_{d,m}^{0}(p,L)\cap{T_{d,m}(p,L)})}\\
  &=1+\codim_{{Y_{d,m}(p,L)}\cap{T_{d,m}(p,L)}}
  {(W_{d,m}^{0}(p,L)\cap{T_{d,m}(p,L)})}\\
  &=1+\codim_{\Grass{(1,\Mat{(m,n-1)})}}{\Delta^{0}(m,n-1)}
  =2n-m.
 \end{align*}
 If $m$ is divisible by the characteristic of $K$, then by
 Lemma~$\ref{delm}$, we have
 \begin{align*}
  &\codim_{Y_{d,m}(p,L)}{W_{d,m}(p,L)}\\
  &=\codim_{\Grass{(1,\Mat{(m,n-1)})}}{\Delta(m,n-1)}
  =\min{\{n-1,2n-m-1\}},\\
  &\codim_{Y_{d,m}(p,L)}{W_{d,m}^{0}(p,L)}\\
  &=\codim_{\Grass{(1,\Mat{(m,n-1)})}}{\Delta^{0}(m,n-1)}
  =2n-m-1.
 \end{align*}
\end{proof}
\begin{proof}[Proof of Theorem~$\ref{gensm}$]
 (i)\quad
 If $m\geq{2n}$, then $\dim{Y_{d,m}}<\dim{M_{d}}$,
 hence $Y_{F,m}$ is empty
 for general $F\in\Sym^{d}V$.\\
 (ii)\quad
 Let
 $\varPsi_{d,m}:\mathcal{Y}_{d,m}\rightarrow\mathcal{M}_{d}$
 be the morphism of the schemes over $\Spec{\mathbf{Z}}$ whose
 fiber at $\Spec{K}\rightarrow\Spec{\mathbf{Z}}$ is the morphism
 $\psi_{d,m}:Y_{d,m}\rightarrow{M_{d}}$ for any field $K$.
 If $m\leq{2n-1}$ and $m$ is prime to the characteristic of $K$,
 then $\codim_{Y_{d,m}}{W_{d,m}}\geq1$ by
 Proposition~$\ref{cod}$, hence
 $\psi_{d,m}:{Y}_{d,m}\rightarrow{{M}_{d}}$ is dominant.
 Therefore
 $\varPsi_{d,m}:\mathcal{Y}_{d,m}\rightarrow\mathcal{M}_{d}$
 is a dominant morphism for $m\leq{2n-1}$.
 Since $\varPsi_{d,m}$ is a proper morphism, $\varPsi_{d,m}$ is
 surjective, hence $Y_{F,m}$ is non-empty for any field $K$ and
 for any $F\in\Sym^{d}V$.\\
 (iii)\quad
 By \cite[Proposition~10.4]{h}, $Y_{F,m}$ is smooth of dimension
 $2n-m-1$ for $[F]\in{M_{d}\setminus{\psi_{d,m}(W_{d,m})}}$.
 Hence we will show that
 $\psi_{d,m}(W_{d,m})\subsetneqq{M_{d}}$
 is a proper Zariski closed subset.
 It is well-known that the hypersurface $X_{F}$ is smooth for a general
 $F\in\Sym^{d}V$.
 Then $Y_{F,1}$ and $Y_{F,2}$ are smooth, hence
 $\psi_{d,2}(W_{d,2})\subsetneqq{M_{d}}$ is a proper Zariski closed
 subset.
 If $3\leq{m}\leq2n-1$ and $m$ is prime to the characteristic of
 $K$, then $\dim{W_{d,m}^{0}}<\dim{M_{d}}$ by
 Proposition~$\ref{cod}$, hence
 ${\psi_{d,m}(W_{d,m})}=
 \psi_{d,2}(W_{d,2})\cup{\psi_{d,m}(W_{d,m}^{0})}
 \subsetneqq{M_{d}}$ is a proper Zariski closed subset.\\
 (iv)\quad
 We assume $(n,m)\neq(2,2)$.
 If ${m}\leq{2n-2}$ and $m$ is prime to the characteristic of
 $K$, then $\codim_{Y_{d,m}}{W_{d,m}}\geq2$ by
 Proposition~$\ref{cod}$.
 Using the same argument as the proof of \cite[Chapter~V.~(4.3.3)]{k},
 the general fiber of
 $\varPsi_{d,m}:\mathcal{Y}_{d,m}\rightarrow\mathcal{M}_{d}$ is
 connected for ${m}\leq{2n-2}$.
 By Zariski's Main Theorem, $Y_{F,m}$ is connected for any field
 $K$ and for any $F\in\Sym^{d}V$.
 We assume $(n,m)=(2,2)$.
 We denote by $X_{d}\rightarrow{M_{d}}$ the universal family of
 curves of degree $d$ in $\mathbf{P}^{2}$.
 Then the natural projection $\phi:Y_{d,2}\rightarrow{X_{d}}$ is
 a birational projective morphism.
 By Zariski's Main Theorem, any fiber of $\phi$ is connected.
 Since $X_{F}$ is connected, $Y_{F,2}$ is connected.
\end{proof}
\begin{theorem}\label{cub}
 Assume that $n\geq2$ and the characteristic of $K$ is not $3$.
 For a cubic form $F\in\Sym^{3}V$, the variety $Y_{F,3}$ is smooth of
 dimension $2n-4$ if and only if $X_{F}$ is a smooth hypersurface in
 $\mathbf{P}^{n}$.
\end{theorem}
\begin{proof}
 We assume that $X_{F}$ is not a smooth hypersurface.
 For $p\in\Sing{X_{F}(\bar{K})}$, there is a line $L$
 in $\mathbf{P}^{n}_{\bar{K}}$ such
 that $(p,L)\in{Y_{F,3}}(\bar{K})$.
 Then $([F],p,L)\in{W_{3,3}(\bar{K})}$, hence
 $Y_{F,3}$ is  not smooth of dimension $2n-4$.
 Conversely, we assume that $Y_{F,3}$ is  not smooth of dimension $2n-4$.
 There is a pair $(p,L)\in{Y_{F,3}}(\bar{K})$ such
 that $([F],p,L)\in{W_{3,3}(\bar{K})}$.
 If $a_{3}\neq0$, then
 $p\in\Sing{X_{F}(\bar{K})}$.
 If $a_{3}=0$ and $p\notin\Sing{X_{F}(\bar{K})}$
 then
 $$
 \rank{
 \begin{pmatrix}
  a_{0,2}&\cdots&a_{0,n}&0&\cdots&0\\
  a_{1,2}&\cdots&a_{1,n}&a_{0,2}&\cdots&a_{0,n}\\
  a_{2,2}&\cdots&a_{2,n}&a_{1,2}&\cdots&a_{1,n}\\
 \end{pmatrix}}
 <3.
 $$
 There exist $\alpha,\beta\in{\bar{K}}$ such that
 $a_{2,j}=\alpha{a_{0,j}}+\beta{a_{1,j}}$
 and $a_{1,j}=\beta{a_{0,j}}$
 for $2\leq{j}\leq{n}$.
 Let $s\in\bar{K}$ be satisfying
 $$
 s^{2}+\beta{s}+(\alpha+\beta^{2})=0.
 $$
 Then $F(s,1,0,\dots,0)=0$ and
 $\frac{\partial{F}}{\partial{x_{j}}}(s,1,0,\dots,0)=0$ for
 $0\leq{j}\leq{n}$,
 hence $X_{F}$ is not a smooth hypersurface.
\end{proof}

\bigskip
\begin{flushleft}
 \textsc{Graduate School of Science\\
 Osaka University\\
 Toyonaka, Osaka, 560-0043\\
 Japan}\\
 \textit{E-mail address}:
 {\ttfamily atsushi@math.sci.osaka-u.ac.jp}
\end{flushleft}

\begin{thebibliography}{1}
 \bibitem{ak}A.~Altman and S.~Kleiman,
	 \textit{Foundations of the theory of Fano schemes,}
	 Compos.\ Math.\ \textbf{34} (1977), 3--47.
 \bibitem{bv}W.~Barth and A.~Van~de~Ven,
	 \textit{Fano-Varieties of lines on hypersurfaces,}
	 Arch.\ Math.\ (Basel) \textbf{31} (1978), 96--104.
 \bibitem{cg}H.~Clemens and P.~Griffiths,
	 \textit{The intermediate Jacobian of the cubic
	 threefold,}
	 Ann.\ of Math.\ (2) {\bfseries 95} (1972), 281--356.
 \bibitem{h}R.~Hartshorne,
	 \textit{Algebraic geometry,}
	 Springer-Verlag, GTM {\bfseries 52} (1977).
 \bibitem{i}A.~Ikeda,
	 \textit{The varieties of intersections of lines and
	 hypersurfaces in projective spaces,}
	 ``Higher dimensional algebraic varieties and vector bundles,''
	 RIMS K\^{o}ky\^{u}roku Bessatsu {\bfseries B9} (2008),
	 115--125.
 \bibitem{k}J.~Koll\'{a}r,
	 \textit{Rational curves on algebraic varieties,}
	 Springer-Verlag,
	 Ergebnisse der Math.\ (3) {\bfseries 32} (1996).
 \bibitem{t}A.~Tjurin,
	 \textit{The geometry of the Fano surface of a nonsingular cubic
	 $F\subset{P^{4}}$ and Torelli theorems for Fano surface and
	 cubics,}
	 Math.\ USSR Izv.\  {\bfseries 5} (1971), 517--546.
\end{thebibliography}
\end{document}